\documentclass[table]{amsart}

\makeatletter
\newcommand{\mylabel}[2]{#2\def\@currentlabel{#2}\label{#1}}
\makeatother
\usepackage{amssymb,stmaryrd,mathrsfs}
\usepackage{comment}
\usepackage{float}
\usepackage{longtable} % for 'longtable' environment
\usepackage{pdflscape} % for 'landscape' environment
\usepackage{orcidlink}
\usepackage{booktabs}
\usepackage{hyperref}
\definecolor{vegasgold}{rgb}{0.77, 0.7, 0.35}
\definecolor{darkgoldenrod}{rgb}{0.72, 0.53, 0.04}
\definecolor{gold(metallic)}{rgb}{0.83, 0.69, 0.22}
\hypersetup{
 colorlinks=true,
 linkcolor=darkgoldenrod,
 filecolor=brown,      
 urlcolor=gold(metallic),
 citecolor=darkgoldenrod,
 pdftitle={Iwasawa theory and ranks of elliptic curves in quadratic twist families},
 }

\usepackage[all,cmtip]{xy}

\usepackage[margin=1.2 5in]{geometry}

\usepackage{tikz}
\usetikzlibrary{shapes.geometric}
\tikzset{every loop/.style={min distance=10mm,looseness=10}}

\DeclareFontFamily{U}{wncy}{}
\DeclareFontShape{U}{wncy}{m}{n}{<->wncyr10}{}
\DeclareSymbolFont{mcy}{U}{wncy}{m}{n}
\DeclareMathSymbol{\Sh}{\mathord}{mcy}{"58}
\usepackage[T2A,T1]{fontenc}
\usepackage[OT2,T1]{fontenc}

\newtheorem{theorem}{Theorem}[section]
\newtheorem{lemma}[theorem]{Lemma}

\newtheorem{conj}[theorem]{Conjecture}

\newtheorem{proposition}[theorem]{Proposition}
\newtheorem{corollary}[theorem]{Corollary}
\newtheorem{definition}[theorem]{Definition}
\newtheorem{lthm}{\color{darkgoldenrod} \textbf{Theorem}} % theorems with letters (for intro)

\numberwithin{equation}{section}
\newcommand{\ralg}{\mathbf{r}_{\op{alg}}}
\newcommand{\ran}{\mathbf{r}_{\op{an}}}
\theoremstyle{remark}
\newtheorem{remark}[theorem]{Remark}
\newtheorem{example}[theorem]{Example}
\newcommand{\op}[1]{\operatorname{#1}}

\newcommand{\Gal}{\operatorname{Gal}}

\newcommand{\Z}{\mathbb{Z}}

\newcommand{\Q}{\mathbb{Q}}
\newcommand{\F}{\mathbb{F}}

\newcommand{\cC}{\mathcal{C}}
\newcommand{\cK}{\mathcal{K}}

\newcommand\mtx[4] { \left( {\begin{array}{cc}
 #1 & #2 \\
 #3 & #4 \\
 \end{array} } \right)}

\begin{document}
\title[Iwasawa Theory and Ranks of Elliptic Curves in Quadratic Twist Families]{Iwasawa Theory and Ranks of Elliptic Curves in Quadratic Twist Families}

\author[J.~Hatley]{Jeffrey Hatley\, \orcidlink{0000-0002-6883-1316}}
\address[Hatley]{
Department of Mathematics\\
Union College\\
Bailey Hall 202\\
Schenectady, NY 12308\\
USA}
\email{hatleyj@union.edu}

\author[A.~Ray]{Anwesh Ray\, \orcidlink{0000-0001-6946-1559}}
\address[Ray]{Chennai Mathematical Institute, H1, SIPCOT IT Park, Kelambakkam, Siruseri, Tamil Nadu 603103, India}
\email{anwesh@cmi.ac.in}

\begin{abstract}
We study the distribution of ranks of elliptic curves in quadratic twist families using Iwasawa-theoretic methods, contributing to the understanding of Goldfeld's conjecture. Given an elliptic curve $ E/\mathbb{Q} $ with good ordinary reduction at $ 2 $ and $\lambda$-invariant $ \lambda_2(E/\mathbb{Q}) = 0 $, we use Matsuno's Kida-type formula to construct quadratic twists $ E^{(d)} $ such that $ \lambda_2(E^{(d)}/\mathbb{Q}) $ remains unchanged or increases by $ 2 $. When the root number of $E^{(d)}$ is $-1$ and the Tate-Shafarevich group $ \Sh(E^{(d)}/\mathbb{Q})[2^\infty] $ is finite, this yields quadratic twists with Mordell--Weil rank $ 1 $. These results support the conjectural expectation that, on average, half of the quadratic twists in a family have rank $ 0 $ and half have rank $ 1 $. In the cases we consider we obtain asymptotic lower bounds for the number of twists by squarefree numbers $d\leq X$ which match with the conjectured value up to an explicit power of $\log X$. Our results also apply to twists by a given prime number and are effective. They complement recent groundbreaking results of Smith on Goldfeld's conjecture.
\end{abstract}

\subjclass[2010]{11G05, 11R23 (primary); 11R45 (secondary).}
\keywords{Kida's formula, Iwasawa theory, Selmer groups, elliptic curves.}

\maketitle

\section{Introduction}
\label{section:intro}
\par Given an elliptic curve $E/\Q$, the Mordell-Weil theorem assures that the group of rational points $E(\Q)$ is finitely generated, so that 
\[
E(\Q) \simeq \Z^{\ralg(E)} \oplus T
\]
for some finite group $T$ and some non-negative integer $\ralg(E)$, called the (algebraic) rank of $E$. The distribution of the ranks $\ralg(E)$ as $E$ varies over all elliptic curves has motivated an enormous amount of research over the past few decades, yet many of their basic properties are still not well understood. For instance, until very recently, a folklore conjecture held that the ranks of elliptic curves over $\Q$ are unbounded; however, recent heuristics have challenged this belief~\cite{bklhpr,alr-ranks}.

While the existence of an upper bound for ranks has become highly-debated, mathematicians still largely believe that, whether elliptic curves of arbitrarily large rank exist, they are extremely rare. More precisely, a conjecture originating with Katz and Sarnak~\cite{katz-sarnak} asserts that, asymptotically, half of all elliptic curves over $\Q$ have rank 0 and half have rank 1. It is thus reasonable to say that, conjecturally, the \textit{average rank} of all elliptic curves over $\Q$ is $1/2$. Recent progress on this conjecture has been encouraging; for instance, Bhargava and Shankar~\cite{bhargava-shankar} prove that the average rank is indeed bounded, and that it is bounded by at most $3/2$.

One may also study the average ranks of elliptic curves in restricted families. For instance, if $E:y^2=f(x)$ is an elliptic curve defined over $\Q$, and if $d$ is a squarefree integer, then the corresponding \textit{quadratic twist} $E^{(d)}:dy^2=f(x)$ is also an elliptic curve defined over $\Q$, and there is an isomorphism $E_{/K} \simeq E^{(d)}_{/K}$ upon base change to $K=\Q(\sqrt{d})$. It is natural to study the variation of arithmetic invariants in quadratic twist families, and a conjecture due to Goldfeld~\cite{goldfeld} predicts that the average rank in quadratic twist families is again $1/2$. (See Section~\ref{sec:goldfeld} for a precise statement of Goldfeld's conjecture.)

As with the Katz-Sarnak conjecture, progress on Goldfeld's conjecture has been modest. Many authors have instead studied the weaker conjecture that \textit{some positive proportion} of elliptic curves in a quadratic twist family have rank 0, and a positive proportion have rank 1. In the notation of Section~\ref{sec:goldfeld}, this is the assertion that
\[
n_{E,r}^*(X)\gg X
\]
for $r\in\{0,1\}.$

When $r=0$, Ono and Skinner~\cite{OnoSkinner} prove that 
\[
n_{E,0}^\ast(X)\gg \frac{X}{\log X},
\] and when $E(\Q)[2]=0$, Ono~\cite{OnoCrelle} improves this to 
\[
n_{E,0}^\ast(X)\gg \frac{X}{\left(\log X\right)^{1-\alpha}}
\]
for some $\alpha \in (0,1)$ depending on $E$. (See Theorems~\ref{Ono Skinner thm} and ~\ref{thm:ono} for the precise statements.) 
\par Kriz and Li~\cite{krizli} have established positive density results for both the rank $0$ and rank $1$ subsets in cases where $E(\mathbb{Q})[2] = 0$. However, their results hinge on certain technical conditions related to the properties of Heegner points that arise in $E(K)$ (see Theorem~\ref{thm:kriz-li} for a detailed statement). These conditions, while theoretically significant, pose challenges when attempting to determine their prevalence in practice.

For rank $0$ quadratic twists of elliptic curves, Kriz and Nordentoft~\cite{kriz2023horizontal} have achieved unconditional positive density results by leveraging analytic techniques involving $p$-adic $L$-functions. Despite this progress, the extent to which their assumptions hold in general remains unclear in practical scenarios.

\par In addition to these approaches, there are several notable conditional results. For instance, Alexander Smith~\cite{smith-selmer} has produced positive density results for elliptic curves with full rational $2$-torsion, contingent on the Birch and Swinnerton-Dyer conjecture. Earlier, Heath-Brown~\cite{heath-brown} obtained related results under the assumption of the generalized Riemann Hypothesis (GRH). More recently, Smith's work~\cite{smith2022distribution} has advanced our understanding of Goldfeld's conjecture, contributing significant breakthroughs to this area. It is worth noting, however, that the methods used in these works differ substantially from ours. Our approach yields effective results that are qualitatively distinct and our methods are different.

\par In the present paper, we contribute to this body of research using a novel Iwasawa-theoretic approach. Kenkichi Iwasawa initiated the study of arithmetic objects over $\Z_p$-extensions of number fields, obtaining striking results on the uniform growth of ideal class groups along such towers; namely, that if $K$ is a number field and
\[
K=K_0\subset K_1\subset \dots \subset K_n \subset K_{n+1}\subset \dots \subset K_\infty=\bigcup_n K_n
\]
is a $\Z_p$-extension for some prime $p$, then there exist constants $c$, $\mu$ and $\lambda$ (depending only on $K_\infty / K)$ such that the $p$-adic valuation of the class number of $K_n$ is given by 
\[
\mu p^n + \lambda n + c
\]
for $n\gg0$. This perspective was generalized to the study of Selmer groups of elliptic curves by Mazur \cite{mazurinventiones}. The details are recalled in Section~\ref{sec:iw-thy-ell}, but we sketch the main ideas here. Given a number field $K$ and a $\Z_p$-extension $K_\infty/K$ as above, one obtains a Selmer group $\op{Sel}_{p^\infty}(E/K_\infty)$ fitting into an exact sequence
\begin{equation}\label{eq:intro-selmer}
0\rightarrow E(K_\infty)\otimes \Q_p/\Z_p\rightarrow \op{Sel}_{p^\infty}(E/K_\infty)\rightarrow \lim_{n\rightarrow \infty}\Sh(E/K_n)[p^\infty]\rightarrow 0.
\end{equation}
Here $\Sh(E/K_n)$ is the Tate-Shafarevich group defined in \eqref{eq:ts}. Moreover, $\op{Sel}_{p^\infty}(E/K_\infty)$ cofinitely generated as a $\Lambda:=\Z_p\llbracket\Gal(K_\infty/K)\rrbracket$-module. If it is furthermore $\Lambda$-cotorsion, then it has associated invariants $\mu_p(E/K)$ and $\lambda_p(E/K)$ which play a role analogous to the constants $\mu$ and $\lambda$ in Iwasawa's results on class groups. In particular, the $\lambda$-invariant gives information on the growth of the rank along the $\Z_p$-extension. In particular, there is a chain of inequalities
\[\ralg(E)\leq \op{corank}_{\Z_p} \op{Sel}_{p^\infty}(E/\Q)\leq \lambda_p(E/\Q).\]
In fact, if $\Sh(E/\Q)[2^\infty]$ is finite, then the first inequality is actually an equality.

Our strategy is to study the variation of the $2$-adic invariants $\lambda_2(E^{(d)}/\Q)$ in quadratic twist families via a Kida-type formula due to Matsuno~\cite{Matsuno-2}. In our context, it follows from results of Kato \cite{kato} and Rubin \cite{RubinBSDCM} that $\op{Sel}_{2^\infty}(E^{(d)}/\Q_\infty)$ is cotorsion over $\Lambda$, cf. Lemma \ref{lemma on Selmer corank and lambda} for further details. Fix an elliptic curve $E/\Q$ with good ordinary reduction at $2$ for which $\lambda_2(E/\Q)=0$. Matsuno's formula (cf. Theorem \ref{thm:kida-matusno-2}) allows us to construct squarefree integers $d$ such that  \[
\lambda_2(E^{(d)}/\Q)=\lambda_2(E/\Q),
\]
from which it follows that $\ralg(E^{(d)})=0$. Similarly, we can construct squarefree integers $d$ such that
\[
\lambda_2(E^{(d)}/\Q)-\lambda_2(E/\Q)=2;
\]
further assumptions on the prime divisors of $d$ then allow us to invoke known cases of the parity conjecture to deduce that $\op{corank}_{\Z_p} \op{Sel}_{p^\infty}(E/\Q)$ is odd. In particular, if $\Sh(E^{(d)}/\Q)[2^\infty]$ is finite, then $\ralg(E^{(d)})=1.$ We also give explicit density results on the sets of admissable squarefree integers $d$ in these constructions.

Our results may be summarized as follows. We note that our results on $\lambda$-invariants and Selmer coranks are unconditional, and that they yield results on ranks (in the direction of Goldfeld's conjecture) when the Shafarevich-Tate group is finite. For $X>0$, let $n_{E, k}'(X)$ denote the number of positive squarefree integers $d<X$ such that $\op{corank}_{\Z_2}\op{Sel}_{2^\infty}(E^{(d)}/\Q)=k$. Denote by $\omega(E)$ the sign of the functional equation of $E$.

\begin{lthm}[Theorem \ref{n'E,1}]\label{thm A}
    Let $E$ be an elliptic curve over $\Q$ and assume that the following conditions are satisfied:
    \begin{enumerate}
        \item $E$ has good ordinary reduction at $2$ with squarefree conductor $N_E$, 
        \item $\omega(E)=-1$,
      \item $E(\Q)[2]=0$,
      \item $\mu_2(E/\Q)=0$ and $\lambda_2(E/\Q)\leq 2$.
    \end{enumerate}
    Then we find that 
    \[n_{E, 1}'(X)\gg \frac{X}{(\log X)^{\frac{11}{12}}}.\]
\end{lthm}
In fact, we are able to obtain an effective version of the above result, see Proposition \ref{corank 1 lemma}. This allows us to pin down squarefree products $d=\ell_1\dots \ell_k$ in terms of Chebotarev conditions on the primes $\ell_i$. In order to demonstrate the effectiveness of this approach, one can also specialize to twist families where $\ell$ is a prime number. We obtain the following result.
\begin{lthm}[Theorem \ref{them prime twist 3.13}]\label{thm a'}
    Let $E$ be an elliptic curve satisfying the conditions of Theorem \ref{thm A}. Let $\mathcal{M}$ be the set of prime numbers $\ell$ such that 
    \begin{itemize}
        \item $\ell\equiv 1\pmod{4}$, 
        \item $\widetilde{E}(\F_\ell)[2]=0$ (where $\widetilde{E}$ denotes the reduced curve at $\ell$), 
        \item $\ell$ splits completely in $\Q(i, \sqrt{-N_E})$.
    \end{itemize}
    Then $\mathcal{M}$ has density $\geq \frac{1}{12}$ and $\op{corank}_{\Z_2}\left( \op{Sel}_{2^\infty}(E^{(\ell)}/\Q)\right)=1$.
\end{lthm}

It is thus natural to wonder how often the conditions of Theorems \ref{thm A} and \ref{thm a'} are satisfied. In order to make sense of this question, we note that any elliptic curve $ E_{/\Q} $ has a unique globally minimal Weierstrass equation. That is, $ E $ is isomorphic to a unique curve $ E_{A,B} $ with minimal Weierstrass equation  
\[
E_{A, B}: y^2 = x^3 + Ax + B,
\]  
where $ A, B \in \Z $ and $\gcd(A^3, B^2)$ is not divisible by $ d^{12} $ for any $ d > 1 $. Let $ \mathcal{C} $ be the set of all such equations. The height of $ E $ is defined as $ \operatorname{ht}(E) := \max\{|A|^3, B^2\} $. For $ X > 0 $, let $ \mathcal{C}(X) $ be the set of all $ E_{A,B} $ with $ \operatorname{ht}(E_{A,B}) \leq X $. Given a set of (isomorphism classes of) elliptic curves $\mathcal{S}\subset \cC$, say that $\mathcal{S}$ has positive density if the limit 
\[\liminf_{X\rightarrow \infty} \frac{\# \{E_{A, B}\in \mathcal{S}\mid \operatorname{ht}(E_{A,B}) \leq X \}}{\# \cC(X)}>0.\]
It is natural to expect that the conditions of Theorem \ref{thm A} are satisfied for a positive density of elliptic curves. Indeed:
\begin{description}
    \item[Condition (1)] depends purely on local properties at all primes. The density of elliptic curves for which $ N_E $ is squarefree is positive; see, for example, \cite{cremonasadek}.
    \item[Condition (2)] pertains to the rank distribution of elliptic curves. By the rank distribution conjecture of Katz and Sarnak, it is expected that half of all elliptic curves satisfy this condition. 
    \item[Condition (3)] holds for almost all elliptic curves $ E/\mathbb{Q} $, as demonstrated by a result of Duke \cite{Dukeexceptional}. 
    \item[Condition (4)] is more subtle and involves analyzing the truncated Euler characteristic of the Selmer group. Specifically, if this Euler characteristic is not divisible by $ p $, the condition holds. Based on heuristic arguments, it is reasonable to expect that condition (4) is satisfied for a proportion of $(1 - 1/p)$ of all elliptic curves $ E/\mathbb{Q} $. However, establishing this rigorously remains an open problem, even for a positive density of curves. Further details on this matter can be found in \cite{kunduraystats1}. In this context, it was shown by the second author \cite{ray2024statisticsIII} that:
\begin{itemize}
    \item $E$ has good ordinary reduction at $5$, 
    \item $\mu_5(E/\Q)=0$ and $\lambda_5(E/\Q)$.
\end{itemize}
\end{description}
\par Our next result applies to elliptic curves for which $E(\Q)[2]\neq 0$.
\begin{lthm}[Theorem \ref{pos density of primes}]\label{thm B}
    Let $E/\Q$ be an elliptic curve satisfying the following conditions:
    \begin{enumerate}
     \item $E$ has good ordinary reduction at $2$ with squarefree conductor $N_E$, 
        \item $\omega(E)=+1$,
        \item $E(\Q)[2]\neq 0$,
        \item $\mu_2(E/\Q)=0$ and $\lambda_2(E/\Q)=0$.
    \end{enumerate}
    Let $\ell$ be a prime satisfying the following conditions:
    \begin{itemize}
        \item $\ell\nmid 2 N_E$, 
        \item $\ell$ is inert in $K:=\Q(\sqrt{-N_E})$,
        \item $\ell\equiv 3, 5\pmod{8}$.
    \end{itemize}
Then we have that $\op{corank}_{\Z_2} \op{Sel}_{2^\infty}(E^{(\ell)}/\Q)=1$. The set of primes satisfying the above conditions have density $\frac{1}{4}$.
\end{lthm}

We also study the distribution of $2$-adic $\lambda$-invariants in quadratic twist families, the following result is proven in section \ref{s 4}. Let $ E $ be an elliptic curve defined over $ \mathbb{Q} $ with good ordinary reduction at $ 2 $ and squarefree conductor $ N_E $. Assume that $ E(\mathbb{Q})[2] = 0 $. Consider the family of quadratic twists $ E^{(d)} $, where $ d > 0 $ is a squarefree integer. Define $ m_{E, N}(X) $ to be the number of such $ d $ (with $ d \leq X $) for which the $ \lambda $-invariant of $ E^{(d)} $ over $ \mathbb{Q} $ is $ N $.

\begin{lthm}[Theorem \ref{last thm}]\label{Thm D}
    Assume that $\mu_2(E/\Q)=0$. If $N$ is any integer such that $N\geq \lambda_2(E/\Q)$ and $N\equiv \lambda_2(E/\Q)\pmod{2}$, then we have that 
    \[m_{E, N}(X)\gg X/(\log X)^{\alpha},\]
    where $\alpha=\begin{cases}
        \frac{1}{3}
        & \text{ if }\op{Gal}(\Q(E[2])/\Q)\simeq \Z/3\Z;\\
        \frac{2}{3}& \text{ if }\op{Gal}(\Q(E[2])/\Q)\simeq S_3.
    \end{cases}$
\end{lthm}

\par It is natural to consider extending our approach to elliptic curves with supersingular reduction at $2$. In this setting, the standard tools used for ordinary reduction are no longer directly applicable since the Selmer group over the cyclotomic $\Z_2$-extension is not $\Lambda$-cotorsion. As a result, understanding the growth of Selmer groups in the cyclotomic $\mathbb{Z}_2$-extension requires alternative frameworks. The $\sharp$ and $\flat$ Selmer groups introduced by Sprung \cite{sprungsharpflat} provide a refined approach to tackling these challenges. However, the theory of $\sharp$ and $\flat$ Selmer groups for $p = 2$ remains underdeveloped compared to the case of odd primes. Important questions about the structure, growth properties, and $2$-adic $L$-functions associated with these Selmer groups are yet to be fully explored. We expect that our work could invigorate interest in such themes. The interplay between the arithmetic of the elliptic curve and the structure of these Selmer groups might also yield deeper insights into $2$-adic families of modular forms and Galois representations. Given the rich structure inherent in supersingular reduction, one can be optimistic that this area will produce intriguing discoveries that connect with broader themes in number theory, such as $p$-adic Hodge theory, Euler systems, and the Birch and Swinnerton-Dyer conjecture in the $2$-adic context.

\section*{Statements and Declarations}

\subsection*{Conflict of interest} The authors report there are no conflict of interest to declare.

\subsection*{Data Availability} There is no data associated to the results of this manuscript.

\section*{Acknowledgements} Partial support for this research was provided by an AMS-Simons Research Enhancement Grant for Primarily Undergraduate Institution Faculty. The second author thanks Union college for hosting him during his visit in October of 2024. The authors thank Tristan Phillips and Asbj\o rn Christian Nordentoft for their helpful comments.

\section{Preliminary notions}
\subsection{Iwasawa theory of elliptic curves}\label{sec:iw-thy-ell}
\par In this section, we recall preliminary notions pertaining to the Iwasawa theory of elliptic curves. First, we must set up some notation. For a field $F$ of characteristic zero, $\op{G}_F$ will denote the absolute Galois group $\op{Gal}(\bar{F}/F)$. Given a prime number $\ell$, choose an embedding $\iota: \bar{\Q}\hookrightarrow \bar{\Q}_\ell$ and let $\op{G}_\ell$ denote the absolute Galois group $\op{Gal}(\bar{\Q}_\ell/\Q_\ell)$. Set $\op{I}_\ell$ to be the inertia subgroup of $\op{G}_\ell$ and let $\op{Frob}_\ell\in \op{G}_\ell/\op{I}_\ell$ be a Frobenius element. The embedding $\iota_\ell$ induces an inclusion $\iota_\ell^*: \op{G}_\ell\hookrightarrow \op{G}_{\Q}$. Throughout, $p$ will be a fixed prime number and $E:y^2=f(x)$ an elliptic curve over $\Q$. Let $\chi$ denote the $p$-adic cyclotomic character and $\omega$ the mod-$p$ reduction of $\chi$. We shall assume throughout that $E$ has potentially good ordinary reduction at $p$. In other words, there is a finite extension $\cK/\Q_p$ such that the base change $E_{/\cK}$ has good ordinary reduction. Given a prime number $\ell$, take $\F_\ell$ to be the field with $\ell$ elements. If $E$ has good reduction at $\ell$, we will denote by $\widetilde{E}(\F_\ell)$ the $\F_\ell$-rational points of the reduction of $E$ at $\ell$. We shall set $E[p^k]$ to denote the $p^k$ torsion points of $E(\bar{\Q})$. Set $E[p^\infty]$ to be the $p$-divisible group of $p$-primary torsion points. Note that $E[p^\infty]$ comes equipped with a natural action of $\op{G}_{\Q}$. Set $\mathbb{T}_p(E)$ to denote the Tate-module, i.e. the inverse limit $\varprojlim_n E[p^n]$ where the inverse limit is taken with respect to multiplication by $p$ maps $\times p: E[p^{n+1}]\rightarrow E[p^n]$. Choose an isomorphism $\mathbb{T}_p(E)\xrightarrow{\sim} \Z_p\oplus \Z_p$, and let \[\rho_{E,p}: \op{G}_{\Q}\rightarrow \op{Aut}_{\Z_p}(\mathbb{T}_p(E))\xrightarrow{\sim} \op{GL}_2(\Z_p)\]
be the associated representation. The mod-$p$ reduction of $\rho_{E,p}$ is denoted 
\[\bar{\rho}_{E,p}:\op{G}_{\Q}\rightarrow \op{GL}_2(\F_p)\] and is identified with the Galois representation on $E[p]$. The condition that $E$ has potentially good ordinary reduction at $p$ implies that there is a finite extension $\cK/\Q_p$ such that $\rho_{E,p}$ restricts to the inertia subgroup of $\op{G}_{\cK}$ to a representation of the form $\mtx{\chi}{\ast}{0}{1}$. Let $d$ be a squarefree integer, denote by $E^{(d)}$ the quadratic twist given by $dy^2=f(x)$.

Given a number field $K$, let $E(K)$ denote the Mordell--Weil group of $E$ over $K$. Let $\Omega_K^f$ be the set of finite primes of $K$ and $\Omega_K^\infty$ the archimedian places. We let $\Omega_K$ be the set of all places of $K$, i.e., $\Omega_K:=\Omega_K^f\cup \Omega_K^\infty$. We shall use $H^i(L, \cdot):=H^i(\op{G}_L, \cdot)$ where $L$ denotes $K$ or $K_v$. The Tate--Shafarevich group $\Sh(E/K)$ is defined as
\begin{equation}\label{eq:ts}
\Sh(E/K):=\op{ker}\left\{H^1(K, E)\rightarrow \prod_{v\in \Omega_K} H^1(K_v, E)\right\}.\end{equation}
It is conjectured that the Tate--Shafarevich group $\Sh(E/K)$ is finite. This is known in some special cases. Rubin \cite{RubinShafinite} proved that if $E$ is a CM elliptic curve with analytic rank at most $1$, then $\Sh(E/\Q)$ is finite. Kolyvagin \cite{Kolyvagin} extended this to all elliptic curves $E_{/\Q}$ of analytic rank at most $1$.
\par Mazur \cite{mazurinventiones} initiated the Iwasawa theory of elliptic curves by studying growth properties of Selmer groups in cyclotomic $\Z_p$-extensions. Selmer groups consist of Galois cohomology classes which satisfy local conditions at all primes. Let us recall their definition and properties in some detail. From the Kummer sequence of $\op{G}_{L}$-modules
\[0\rightarrow E(\bar{L})[p^n]\rightarrow E(\bar{L})[p^\infty]\xrightarrow{\times p^n} E(\bar{L})[p^\infty]\rightarrow 0,\] we obtain the Kummer map 
\[\kappa_{E,L}^{(n)}: \frac{E(L)}{p^n E(L)}\hookrightarrow H^1(L, E[p^n]). \] Passing to the direct limit in $n$, one obtains:
\begin{equation}\label{kapps E L defn}\kappa_{E,L}: E(L)\otimes \Q_p/\Z_p\hookrightarrow H^1(L, E[p^\infty]).\end{equation}
\begin{definition}
    The $p$-primary Selmer group $\op{Sel}_{p^\infty}(E/K)$ consists of all classes $\alpha\in H^1(K, E[p^\infty])$ such that $\alpha_{|v}\in \op{image}\kappa_{E, K_v}$ for all primes $v\in \Omega_K$. 
\end{definition}
The Selmer group $\op{Sel}_{p^\infty}(E/K)$ is cofinitely generated as a $\Z_p$-module, and it fits into a natural short exact sequence
\begin{equation}\label{ses} 0\rightarrow E(K)\otimes \Q_p/\Z_p\xrightarrow{\kappa_{E,K}} \op{Sel}_{p^\infty}(E/K)\rightarrow \Sh(E/K)[p^\infty]\rightarrow 0.\end{equation}
Let $K_\infty$ denote the cyclotomic $\Z_p$-extension of $K$, i.e., the unique extension contained in $K(\mu_{p^\infty})$ for which $\op{Gal}(K_\infty/K)\xrightarrow{\sim} \Z_p$ as a topological group. We set $\Gamma:=\op{Gal}(K_\infty/K)$ and choose a topological generator $\gamma\in \Gamma$. Let $K_n$ be the $n$-th layer of $K_\infty$, defined to be the subextension of $K_\infty$ such that $[K_n:K]=p^n$. View
\[K=K_0\subset K_1\subset \dots \subset K_n \subset K_{n+1}\subset \dots \subset K_\infty=\bigcup_n K_n\]
as a tower of extensions and identify $\op{Gal}(K_n/K)$ with $\Gamma_n:=\Gamma/\Gamma^{p^n}$.\begin{definition}
    With respect to notation above, the Iwasawa algebra is then taken to be the completed group ring $\Lambda:=\varprojlim_n \Z_p[\Gamma/\Gamma^{p^n}]$.
\end{definition} 
Setting $T:=(\gamma-1)\in \Lambda$, we identify $\Lambda$ with the formal power series ring $\Z_p\llbracket T\rrbracket$. A monic polynomial $f(T)\in \Z_p\llbracket T\rrbracket$ is said to be a \emph{distinguished polynomial} if its nonleading coefficients are all divisible by $p$. Let $M$ be a module over $\Lambda$, set $M^\vee:=\op{Hom}_{\Z_p}(M, \Q_p/\Z_p)$ to be the Pontryagin dual of $M$. Then $M$ is said to be \emph{cofinitely generated} (resp. \emph{cotorsion}) over $\Lambda$ if $M^\vee$ is finitely generated (resp. torsion) over $\Lambda$. Let $M$ and $M'$ be cofinitely generated and cotorsion $\Lambda$-modules. Then, $M$ and $M'$ are said to be \emph{pseudo-isomorphic} if there is a map $\phi: M\rightarrow M'$ of $\Lambda$-modules, whose kernel and cokernel are finite. By the structure theory of finitely generated and torsion $\Lambda$-modules, any cofinitely generated and cotorsion $\Lambda$-module $M$ is pseudo-isomorphic to a module $M'$ whose Pontryagin dual is given as a direct sum of cyclic torsion $\Lambda$-modules 
\begin{equation}\label{cyclic isomorphism}(M')^\vee\simeq \left(\bigoplus_{i=1}^s \frac{\Lambda}{(p^{n_i})}\right)\oplus \left(\bigoplus_{j=1}^t \frac{\Lambda}{(f_j(T))}\right).\end{equation} Here, $s$ and $t$ are natural numbers (possibly $0$), $n_i \in \Z_{\geq 1}$ and $f_j(T)$ are distinguished polynomials. Then, one defines the Iwasawa $\mu$- and $\lambda$-invariants by 
\[\mu_p(M):=\sum_{i=1}^s n_i \quad \text{ and } \quad \lambda_p(M):= \sum_{j=1}^t \op{deg}f_j.\]
Here, if $s=0$ (resp. $t=0$), the sum $\sum_{i=1}^s n_i$ (resp. $\sum_{j=1}^t \op{deg}f_j$) is interpreted as being equal to $0$.

\begin{lemma}
    Let $M$ be a cofinitely generated and cotorsion $\Lambda$-module. Then $\mu_p(M)=0$ if and only if $M$ is cofinitely generated as a $\Z_p$-module, i.e., $M\simeq (\Q_p/\Z_p)^\lambda\oplus T$ where $T$ is finite. The quantity $\lambda$ is the $\Z_p$-corank of $M$ and is given by $\lambda=\lambda_p(M)$.
\end{lemma}
\begin{proof}
    The result above is an easy consequence of the structural decomposition \eqref{cyclic isomorphism} and its proof is omitted.
\end{proof}

\par The Selmer group $\op{Sel}_{p^\infty}(E/K_\infty)$ is defined to be the direct limit $\varinjlim_n \op{Sel}_{p^\infty} (E/K_n)$ and is a cofinitely generated module over $\Lambda$. Moreover, when $E$ has good ordinary reduction at $p$ and $K/\Q$ is an abelian extension, it is known due to results of Kato \cite{kato} and Rubin \cite{RubinBSDCM} that $\op{Sel}_{p^\infty}(E/K_\infty)$ is cotorsion over $\Lambda$. 
\begin{definition}If $E$ is an elliptic curve over $\Q$ and $K$ is a number field for which $\op{Sel}_{p^\infty}(E/K_\infty)$ is cotorsion over $\Lambda$, we write $\mu_p(E/K)$ and $\lambda_p(E/K)$ to denote the associated $\mu$-invariant (resp. $\lambda$-invariant).
\end{definition}
It is conjectured by Greenberg \cite[Conjecture 1.11]{GreenbergITEC} that if $E_{/\Q}$ is an elliptic curve with good ordinary reduction at $p$ such that $\rho_{E, p}$ is irreducible, then $\mu_p(E/\Q)=0$. We end this subsection with a criterion due to Greenberg for the entire Selmer group $\op{Sel}_{p^\infty}(E/\Q_\infty)$ to vanish. At a prime number $\ell$, let $c_\ell(E)$ be the \emph{Tamagawa number} defined as  $c_\ell(E):=[\mathscr{E}(\Q_\ell): \mathscr{E}_0(\Q_\ell)]$, where $\mathscr{E}_0$ is the connected component of the Neron model $\mathscr{E}$ associated to $E$ over $\Q_\ell$.
\begin{proposition}[Greenberg]\label{main prop}
    Let $E$ be an elliptic curve and $p$ be a prime number. Assume that:
    \begin{enumerate}
        \item $E$ has good ordinary reduction at $p$,
        \item $\op{Sel}_p(E/\Q)=0$,
        \item $p\nmid \# \widetilde{E}(\F_p)$. 
        \item $p\nmid c_\ell(E)$ for all primes $\ell\neq p$.
    \end{enumerate}
    Then $\op{Sel}_{p^\infty}(E/\Q_{\infty})=0$, and in particular, 
    \[\mu_p(E/\Q)=0\text{ and }\lambda_p(E/\Q)=0. \]
\end{proposition}
\begin{proof}
    The result follows from \cite[remark following Proposition 3.8, p.80, ll. 19--26]{GreenbergITEC}.
\end{proof}

\subsection{A formula of Hachimori and Matsuno}
\par Hachimori and Matsuno \cite{hachimori-matsuno} proved a formula for the growth of Iwasawa invariants in a finite $p$-extension, generalizing Kida's formula from classical Iwasawa theory. Let $L/K$ be a finite Galois extension whose Galois group $G:=\op{Gal}(L/K)$ is a $p$-group. Let $S_{\op{add}}$ be the set of primes of $K$ at which $E$ has additive reduction. Following Hachimori and Matsuno, in this subsection we temporarily make the following additional assumptions:
\begin{itemize}
\item[(\mylabel{hyp:Hyp1}{\textbf{Hyp1}})] If $p=2$, then $K$ is totally imaginary. 
\end{itemize}
\begin{itemize}
\item[(\mylabel{hyp:Hyp2}{\textbf{Hyp2}})] $E$ has additive reduction at all primes of $L_\infty$ lying above $S_{\op{add}}$. 
\end{itemize}
%\begin{hypothesis}\label{hyp 1}
 %   If $p=2,3$ then we assume in addition that 
  %  \begin{enumerate}
   %     \item if $p=2$ then $K$ is totally imaginary, 
    %    \item $E$ has additive reduction at all primes of $L_\infty$ lying above $S_{\op{add}}$.
   % \end{enumerate}
%\end{hypothesis}
We shall briefly recall their result now.

\begin{theorem}[Hachimori--Matsuno]\label{thm:kida-hachimori-matsuno}
    With respect to the notation above, assume that the following conditions are satisfied:
    \begin{enumerate}
        \item $E$ has good ordinary reduction at $p$ and $K/\Q$ an abelian extension.
        \item The conditions \eqref{hyp:Hyp1} and \eqref{hyp:Hyp2} hold,
        \item $\mu_p(E/K)=0$.
    \end{enumerate}
    The following statements hold:
    \begin{enumerate}
        \item The Selmer group $\operatorname{Sel}_{p^\infty}(E/L_\infty)$ is cofinitely generated and cotorsion over $\Lambda$ with $\mu_p(E/L)=0$.
        \item We have the equation:
   \[
   \lambda_p(E/L)=[L_\infty:K_\infty] \lambda_p(E/K)+\sum_{w\in P_1(E)} \left(e(w)-1\right)+2 \sum_{w\in P_2(E)} \left(e(w)-1\right).
   \]
   
   Here $P_1(E)$ and $P_2(E)$ are sets of primes in $L_\infty$ defined as:
   \[
   \begin{split}
   & P_1(E) := \{w \mid w \nmid p, \text{ $E$ has split multiplicative reduction at } w\}, \\
   & P_2(E) := \{w \mid w \nmid p, \text{ $E$ has good reduction at } w \text{ and } E(L_{\infty, w}) \text{ has a point of order } p\},
   \end{split}
   \]
    \end{enumerate}
    and $e(w)=e_{L_\infty/K_\infty}(w)$ is the ramification index of $w$ over $K_\infty$
\end{theorem}
\subsection{Goldfeld's conjecture}\label{sec:goldfeld}
\par Given an elliptic curve $E:y^2=f(x)$ over $\Q$, denote by $\ralg(E)$ the rank of the Mordell--Weil group $E(\Q)$. On the other hand, set $\ran(E)$ to denote the analytic rank, i.e., the order of the zero of the Hasse--Weil L-function $L_E(s)$ at $s=1$. The weak form of the Birch and Swinnerton--Dyer conjecture predicts that $\ralg(E)=\ran(E)$. By the theorems of Gross--Zagier and Kolyvagin, this conjecture is known whenever $\ran(E)\in \{0, 1\}$. It is natural to study the distribution of $\ralg(E)$ and $\ran(E)$ in families of elliptic curves $E_{/\Q}$. A natural family of elliptic curves is the quadratic twist family, which is defined for any elliptic curve $E_{/\Q}$. Given a squarefree integer $d$, let $E^{(d)}$ denote the quadratic twist of $E$ by $d$, given by $E^{(d)}:dy^2=f(x)$. The family of twists $\{E^{(d)}\}$ is ordered according to $|d|$. The parity of $\ran(E)$ is determined by the root number, and as a general principal, one expects that on average, $\ran(E)$ is as small as possible. For $X>0$ and $r\in \Z_{\geq 0}$, set $n_{E,r}^{\ast}(X):=\#\{d\mid \ast(E^{(d)})=r, |d|< X\}$ where $*\in \{\op{alg}, \op{an}\}$ and $d$ refers to a squarefree integer.
\begin{conj}[Goldfeld's conjecture]
    With respect to notation above for $r\in \{0,1\}$,
    \[n_{E,r}^{\ast}(X)\sim \frac{1}{2} \sum_{|d|<X} 1,\] where the sum runs over fundamental discriminants $d$. 
\end{conj}
The conjecture implies that for $r\geq 2$, we have that $n_{E,r}^*(X)=o(X)$. Moreover, it implies that for $r \in \{0,1\}$, we have $n_{E,r}^*(X)\gg X$. We note that the conjecture for $\ast=\op{an}$ implies the conjecture for $\ast=\op{alg}$.

We now recall some of the strongest results shedding light on the conjecture due to Ono--Skinner \cite{OnoSkinner} and Ono \cite{OnoCrelle} which hold for elliptic curves in general. 

\begin{theorem}[Ono--Skinner]\label{Ono Skinner thm}
    Let $E_{/\Q}$ be an elliptic curve, then the following assertions hold:
    \begin{enumerate}
        \item  $n_{E,0}^\ast(X)\gg \frac{X}{\log X}$ for $\ast\in \{\op{an}, \op{alg}\}$.
        \item If the conductor of $E$ is $\leq 100$ then there is a positive density of primes $p$ for which the twist $E^{(p)}$ or $E^{(-p)}$ has rank $0$.
    \end{enumerate}
   \begin{proof}
       Part (1) for $\ast=\op{an}$ is \cite[Corollary 3]{OnoSkinner}. The result for $\ast=\op{alg}$ follows from that for $\ast=\op{an}$. Part (2) is \cite[Corollary 2]{OnoSkinner}.  \end{proof}
\end{theorem}

\begin{theorem}[Ono]\label{thm:ono}
    Let $E$ be an elliptic curve over $\Q$ such that $E(\Q)[2]=0$. Then there is $\alpha(E)\in (0, 1)$ such that $n_{E,0}^\ast(X)\gg \frac{X}{(\log X)^{1-\alpha(E)}}$ for $\ast\in \{\op{an}, \op{alg}\}$.
\end{theorem}
\begin{proof}
    This result is \cite[Corollary 3]{OnoCrelle}.
\end{proof}
Results of a different flavor due to Kriz--Li \cite{krizli} hold for a special class of elliptic curves. Let $E$ be an elliptic curve over $\Q$ with conductor $N$ and $K/\Q$ be an imaginary quadratic field satisfying the Heegner hypothesis. This means that each prime $\ell|N$ splits in $K$. Let $P\in E(K)$ be a Heegner point and $\pi_E: X_0(N)\rightarrow E$ be the modular parametrization. Let $\omega_E\in H^0(E, \Omega^1)$ be such that $\pi_E^*(\omega_E)= f_E(q) \frac{d q}{q}$ where $f_E(q)$ is the modular form associated to $E$. Let $\log \omega_E$ denote the formal logarithm associated to $\omega_E$. Take $\Q(E[2])$ to be the field fixed by the kernel of $\bar{\rho}_{E,2}$ and identify $\op{Gal}(\Q(E[2])/\Q)$ with the image of $\bar{\rho}_{E,2}$.

\begin{theorem}[Kriz--Li]\label{thm:kriz-li}
    With respect to the notation above, assume that:
    \begin{enumerate}
        \item $2$ splits in $K$,
        \item $(E, K)$ satisfies the Heegner hypothesis, 
        \item $\frac{\# \widetilde{E}^{\op{ns}}(\F_2) \cdot \log_{\omega_E}(P)}{2}\not\equiv 0\pmod{2}$ where $\widetilde{E}$ denotes the reduction of $E$ at $2$. 
    \end{enumerate}
    Then for $r\in \{0, 1\}$ and $*\in \{\op{an}, \op{alg}\}$ we have that 
    \[n_{E,r}^*(X)\gg \begin{cases}
        & \frac{X}{(\log X)^{5/6}} \text{ if }\op{Gal}(\Q(E[2])/\Q)\simeq S_3;\\
        & \frac{X}{(\log X)^{2/3}} \text{ if }\op{Gal}(\Q(E[2])/\Q)\simeq \Z/3\Z.
    \end{cases}\]
\end{theorem}
\begin{proof}
    The result is \cite[Theorem 1.12]{krizli}.
\end{proof}

\subsection{Analytic ingredients}
\par In this section, we discuss tauberian theorems and consequences which will come in handy in proving our density results. Let $F(X)$ and $G(X)$ be non-negative functions of a variable $X\in \mathbb{R}_{\geq 0}$. Assume that $G(X)>0$ for large enough values of $X$, Then we write $F(X)\sim G(X)$ to mean that $\lim_{X\rightarrow \infty} \frac{F(X)}{G(X)}=1$. On the other hand, write $F(X)\gg G(X)$ to mean that $F(X)>c G(X)$ for some positive constant $c$ which is independent of $X$. Let us begin with the standard Delange's tauberian theorem.
\begin{theorem}[Delange's tauberian theorem]\label{delange}
        Let $f(s):=\sum_{n=1}^\infty a_n n^{-s}$ be a Dirichlet series with non-negative coefficients and $a>0$ be a real number. Assume that $f(s)$ converges for $\op{Re}(s)>a$ and has a meromorphic continuation to a neighbourhood $U$ of $\op{Re}(s)\geq a$. For $X>0$, we set 
    $g(X):=\sum_{n\leq X} a_n$. Assume that the only pole of $f(s)$ is at $s=a$ and the order of this pole is $b\in \mathbb{R}_{>0}$, i.e., 
    \[f(s)=\frac{1}{(s-a)^b} h(s)\] for some holomorphic function $h(s)$ defined on $U$. Then, there is a positive constant $c>0$ such that, as $X\rightarrow \infty$, we have 
    \[g(X)\sim c X^a (\log X)^{b-1}.\]
\end{theorem}

A set of prime numbers $\Omega$ is said to have a natural density $\mathfrak{d}(\Omega)$ if the following limit exists:
\[\mathfrak{d}(\Omega):=\lim_{X\rightarrow\infty} \frac{\#\{\ell\in \Omega\mid \ell\leq X\}}{\pi(X)}.\]
On the other hand, $\Omega$ is said to have Dirichlet density if the following limit exists
\[\mathfrak{d}'(\Omega):=\lim_{s\rightarrow 1^+}\frac{\sum_{\ell\in \Omega}\ell^{-s}}{\sum_{\ell}\ell^{-s}},\] where in the denominator, the sum is over all prime numbers. If the natural density exists, then so does the Dirichlet density and moreover, $\mathfrak{d}(\Omega)=\mathfrak{d}'(\Omega)$. In this article, the word density for a set of primes shall refer to their natural density. Associated to $\Omega$ is the set of square-free natural numbers $N_\Omega$ consisting of $d$ divisible only by primes $\ell\notin \Omega$. We set $N_\Omega(X):=N_\Omega \cap [1,X]$ and $n_\Omega(X):=\# N_\Omega(X)$. 

\begin{proposition}\label{density propn for n_Omega}
    Let $\Omega$ be a set of primes with positive density $\alpha\in (0, 1)$, then $n_\Omega(X)\sim  c X/(\log X)^{\alpha}$, where $c>0$ is a constant.
\end{proposition}
\begin{proof}
The result follows from \cite[Theorem 2.4, p.5 line -3 to p.6 line -10.]{serredivisibilite}. We sketch the proof here, and for further details, we refer to \emph{loc. cit.} Let $\Omega^c$ be the complement of $\Omega$ and \[\begin{split}f(s)  :=\sum_{n\in N_\Omega} n^{-s} & =\sum_{T\subset \Omega^c}  \left(\prod_{\ell \in T} \ell\right)^{-s} \\
     & = \prod_{\ell\notin \Omega}\left(1+\ell^{-s}\right).\end{split}\]
    It is easy to see that
\[\log f(s)=\sum_{\ell\notin \Omega} \ell^{-s}+k_1(s),\] where $k_1(s)$ is holomorphic on $\op{Re}(s)\geq 1$; and as a consequence,
\[\log f(s) =(1-\alpha) \log\left(\frac{1}{s-1}\right)+k_2(s),  \] where $k_2(s)$ is holomorphic on $\op{Re}(s)\geq 1$. Thus, we deduce that 
\[f(s)=(s-1)^{\alpha-1}h(s),\] where $h(s)$ is a non-zero holomorphic function on $\op{Re}(s)\geq 1$. It follows from the Theorem \ref{delange} that 
\[n_\Omega(X)\sim c X (\log X)^{-\alpha},\] where $c>0$ is a constant that does not depend on $X$.  
\end{proof}

\section{Selmer coranks in quadratic twist families}

As in Section~\ref{sec:goldfeld}, let $E$ be an elliptic curve defined over $\Q$, and for any squarefree integer $d$, let $E^{(d)}$ denote its quadratic twist by $d$. In this section, we seek to study the Iwasawa invariants of $E^{(d)}$ in terms of those of $E$.

%\par For any squarefree value of $d$, $E^{(d)}$ is isomorphic to $E$ over the quadratic extension $K:=\Q(\sqrt{d})$. We shall assume that $E$ has good ordinary reduction at $2$. It then follows that $\op{Sel}_{2^\infty}(E/K_\infty)$ is cotorsion over $\Lambda$. Equivalently, $\op{Sel}_{2^\infty}(E^{(d)}/K_\infty)$ has the same property. In particular, $\op{Sel}_{2^\infty}(E^{(d)}/\Q_\infty)$ is cotorsion over $\Lambda$. The Iwasawa invariants $\mu_2(E^{(d)})$ and $\lambda_2(E^{(d)})$ are well defined. 

\begin{lemma}\label{lemma on Selmer corank and lambda}
    Let $E$ be an elliptic curve with good ordinary reduction at $2$ and $d$ be a squarefree integer. Then the following assertions hold:
    \begin{enumerate}
        \item $\op{Sel}_{2^\infty}(E^{(d)}/\Q_\infty)$ is cotorsion over $\Lambda$. Consequently the Iwasawa invariants $\mu_2(E^{(d)}/\Q)$ and $\lambda_2(E^{(d)}/\Q)$ are well defined. 
        \item $\op{corank}_{\Z_2} \op{Sel}_{2^\infty}(E^{(d)}/\Q)\leq \lambda_2(E^{(d)}/\Q)$.
    \end{enumerate}
\end{lemma}
\begin{proof}
    We note that $E^{(d)}$ is isomorphic to $E$ over the quadratic extension $K:=\Q(\sqrt{d})$. Since $E$ has good ordinary reduction at $2$, the same is true of $E^{(d)}_{/K}$. It then follows that $\op{Sel}_{2^\infty}(E/K_\infty)$ is cotorsion over $\Lambda$. Equivalently, $\op{Sel}_{2^\infty}(E^{(d)}/K_\infty)$ has the same property. In particular, $\op{Sel}_{2^\infty}(E^{(d)}/\Q_\infty)$ is cotorsion over $\Lambda$, thus completing the proof of (1).
    \par It follows from \cite[Theorem 1.1]{Howson} that
    \[\begin{split}& \op{corank}_{\Z_2} H^1(\Gamma, E^{(d)}(\Q_\infty)[2^\infty])\\ =& \op{corank}_{\Z_2} H^0(\Gamma, E^{(d)}(\Q_\infty)[2^\infty]) \\ =& \op{corank}_{\Z_2}E^{(d)}(\Q)[2^\infty]=0.\end{split}\]
    Therefore, $H^1(\Gamma, E^{(d)}(\Q_\infty)[p^\infty])$ is finite. The natural restriction map 
    \[\op{Sel}_{2^\infty}(E^{(d)}/ \Q)\rightarrow \op{Sel}_{2^\infty}(E^{(d)}/\Q_\infty)\]
    has finite kernel since $H^1(\Gamma, E^{(d)}(\Q_\infty)[p^\infty])$ is finite. This proves part (2).
\end{proof}
We note that part (2) of Lemma \ref{lemma on Selmer corank and lambda} implies that $\ralg(E^{(d)})\leq \lambda_2(E^{(d)}/\Q)$. If the $2$-primary part of $\Sh(E^{(d)}/\Q)$ is finite then $\ralg(E)=\op{corank}_{\Z_2} \op{Sel}_{2^\infty}(E^{(d)}/\Q_\infty)$. By aforementioned results of Kato and Rubin, this is known if $E^{(d)}$ has analytic rank at most $1$. We now recall a formula of Matsuno which relaxes the hypotheses of Theorem~\ref{thm:kida-hachimori-matsuno} when $p=2$~\cite[Theorem 5.1]{Matsuno-2}. Although Matsuno's result can be stated more generally, we will use a specialized version where $d$ is assumed to be coprime to the conductor $N_E$ of the elliptic curve $E$. Since we are primarily interested in proving distribution results, this condition on $d$ does not make a significant difference.

\begin{theorem}[Matsuno]\label{thm:kida-matusno-2}Let $E/\Q$ be an elliptic curve with good ordinary reduction at $p=2$ with squarefree conductor $N_E$. Assume that $\mu_2(E/\Q)=0$. Let $d>0$ be a squarefree integer coprime to $N_E$ and let $E^{(d)}$ be the corresponding quadratic twist. Then 
\begin{equation}\label{eq:kida-matsuno-formula}
\lambda_2(E^{(d)}/\Q) = \lambda_2(E/\Q) + \sum_{\substack{\ell \mid d \\ 2 \mid \# \tilde{E}(\F_2)}} 2^{n_\ell + 1}
\end{equation}
where the sum runs over the odd prime divisors of $d$ and where $n_\ell = \mathrm{ord}_2 \left( \frac{\ell^2 - 1}{8} \right)$.
\end{theorem}
We recall from Lemma \ref{lemma on Selmer corank and lambda} that  
\begin{equation}\label{effective bounds after Matsuno}  
\ralg(E) \leq \operatorname{corank}_{\mathbb{Z}_2} \operatorname{Sel}_{2^\infty}(E^{(d)}/\mathbb{Q}) \leq \lambda_2(E^{(d)}/\mathbb{Q}).  
\end{equation}  

Our strategy will be to control the parity of the middle term for a specific subfamily of twists $d$. As a consequence of Theorem \ref{thm:kida-matusno-2}, we can choose these twists such that $\lambda_2(E^{(d)}/\mathbb{Q})$ is effectively bounded. Let $\omega(E)$ denote the sign of the functional equation of the Hasse–Weil $L$-function of $E$ over $\mathbb{Q}$. The $p$-parity conjecture over $\Q$ is known due to Dokchitser and Dokchitser \cite[Theorem 1.4]{dok-dok}. This will be central to our approach.
\begin{theorem}[Dokchitser--Dokchitser]\label{dok-dok thm}
    Let $p$ be a prime number. Then the quantity $s_p(E):=\op{corank}_{\Z_p}\op{Sel}_{p^\infty}(E/\Q)$ is even if and only if $\omega(E)=+1$.
\end{theorem}
As a result for $E^{(d)}$, $\op{corank}_{\Z_2} \op{Sel}_{2^\infty}(E/\Q)=1$ follows as a consequence of bounding $\lambda_2(E^{(d)}/\Q)$ by $2$ and the parity of the root number $\omega(E)$. 
\begin{corollary}
    Suppose that $E^{(d)}$ is a twist of $E$ such that:
    \begin{enumerate}
        \item $\lambda_2(E^{(d)}/\Q)\leq 2$, 
        \item $\omega(E^{(d)})=-1$. 
    \end{enumerate}
    Then we find that $\op{corank}_{\Z_2} \op{Sel}_{2^\infty}(E/\Q)=1$. 
\end{corollary}

\begin{proof}
    Note that Theorem \ref{dok-dok thm} implies that $s_2(E)$ is odd. On the other hand, Lemma \ref{lemma on Selmer corank and lambda} asserts that \[s_2(E)\leq \lambda_2(E^{(d)}/\Q)\leq 2.\] We thus deduce that $s_2(E)=1$.
\end{proof}
Throughout the rest of this section, fix an elliptic curve $E_{/\Q}$ which satisfies the following conditions:
\begin{itemize}
    \item $E$ has good ordinary reduction at $p=2$, 
    \item the conductor $N_E$ is squarefree,
    \item $\mu_2(E/\Q)=0$.
\end{itemize}
Note that since $E$ is assumed to have good ordinary reduction, the $\mu$ and $\lambda$-invariants are indeed well-defined. 
\begin{definition}
    Let $\Omega=\Omega_E$ be the set of rational primes $\ell$ such that $\ell\nmid 2N_E$ and $\widetilde{E}(\F_\ell)[2]\neq 0$.
\end{definition}
Given a prime $\ell\nmid 2N_E$, we find that $\ell\in \Omega$ if and only if $\op{trace}\left(\bar{\rho}_{E,2}(\op{Frob}_\ell)\right)$ is even. Thus by the Chebotarev density theorem, the set $\Omega$ has positive density determined by the image of $\bar{\rho}_{E,2}$. This image is determined as follows. We note that $\op{GL}_2(\F_2)\simeq S_3$ and up to conjugacy, there are only three proper subgroups of $\op{GL}_2(\F_2)$:
\begin{itemize}
    \item $G_1=\left\{\mathbf{1}\right\}$, 
    \item $G_2=\left\{\mathbf{1}, \mtx{1}{1}{0}{1}\right\}$,
    \item $G_3=\left\{\mathbf{1}, \mtx{1}{1}{1}{0}, \mtx{0}{1}{1}{1}\right\}$.
\end{itemize}
Note that $E(\Q)[2]\neq 0$ if and only if $\op{image}\bar{\rho}_{E,2}$ is conjugate to $G_1$ or $G_2$.
\begin{theorem}[Zywina]
    Set
    \[
    J_1 = 256 \frac{(t^2+t+1)^3}{t^2(t+1)^2}, \quad \quad J_2(t)=256 \frac{(t+1)^3}{t}, \quad  \quad J_3(t) = t^2 + 1728,
    \]
    and let $E_{/\Q}$ be an elliptic curve without complex multiplication. Then the image of $\bar{\rho}_{E,2}$ is conjugate to a subgroup of $G_i$ if and only its $j$-invariant $j(E)$ is of the form $J_i(t)$ for some $t\in \Q$.
\end{theorem}

\begin{remark}\label{trivial remark}
    There are 3 cases to consider:
    \begin{itemize}
        \item if $E(\Q)[2]\neq 0$, then all primes $\ell\nmid 2 N_E$ are contained in $\Omega$.
        \item If the image of $\bar{\rho}_{E, 2}$ is conjugate to $G_3$, then a prime $\ell\nmid 2 N_E$ is contained in $\Omega$ if and only if $\bar{\rho}_{E, 2}(\op{Frob}_\ell)=\op{Id}$.
        \item If $\bar{\rho}_{E,2}$ is surjective then $\ell\in \Omega$ if and only if $\bar{\rho}_{E, 2}(\op{Frob}_\ell)\notin \left\{\mtx{1}{1}{1}{0}, \mtx{0}{1}{1}{1}\right\}$.
    \end{itemize}
\end{remark}

\begin{corollary}\label{cor:omega-density-p-2}
Let $E/\Q$ be an elliptic curve without complex multiplication, and let $j(E)$ denote its $j$-invariant. If $\bar{\rho}_{E,2}$ is not surjective, let $i \in \{1, 2, 3\}$ be the smallest value such that $j(E)=J_i(t)$ for some $t \in \Q$.  Then the density of $\Omega$ is given by 
\[
\mathfrak{d}(\Omega)=\begin{cases}
1, & \text{if}\ i=1,2\\
\frac{1}{3}, & \text{if i=3}
\end{cases}.
\]
On the other hand, if $\bar{\rho}_{E,2}$ is surjective, $\mathfrak{d}(\Omega)=\frac{2}{3}$.
\end{corollary}
\begin{proof}
    For a prime $\ell \nmid N_E$, let $\mathrm{Frob}_\ell$ denote the corresponding Frobenius element. Then we have 
\[
a_\ell(E) \equiv \mathrm{tr}\ \bar{\rho}_{E,2}\left( \mathrm{Frob}_\ell\right) \mod 2.
\]
As the trace is conjugacy invariant, we can thus investigate the frequency that $E(\F_\ell)[2]\neq 0$ in terms of the $G_i$. The result follows as a consequence of Remark \ref{trivial remark} and the Chebotarev density theorem.
\end{proof}

\begin{corollary}\label{cor 3.9}
    Let $d$ be a positive squarefree integer and write $d=\ell_1\dots \ell_k$. Assume that for all $i=1,\ldots,k$ we have  $\ell_i\nmid 2 N_E$ and $\ell_i\notin \Omega$. Then we find that $\lambda_2(E^{(d)}/\Q)=\lambda_2(E/\Q)$. Moreover, we find that \[\op{corank}_{\Z_2} \op{Sel}_{2^\infty}(E^{(d)}/\Q)\leq \lambda_2(E/\Q).\]
\end{corollary}
\begin{proof}
    The result is a direct consequence of Theorem \ref{thm:kida-matusno-2} and Lemma \ref{lemma on Selmer corank and lambda}.
\end{proof}
Let $\Omega'$ be the set of primes $\ell \nmid 2 N_E$ such that $\ell\notin \Omega$. According to Remark \ref{trivial remark}, when $E(\Q)[2]\neq 0$, the set $\Omega'$ is finite. On the other hand, by Corollary~\ref{cor 3.9}, if $E(\Q)[2]=0$ then the density of $\Omega'$ is $2/3$ (resp. $1/3$) if the image of $\bar{\rho}_{E,2}$ is conjugate to $G_3$ (resp. $\op{GL}_2(\F_2)$). Given a natural number $r$, let $n^{\op{alg}}_{E,\leq r}(X)$ be the number of squarefree numbers $d>0$ such that $\textbf{r}_{\op{alg}}(E^{(d))})\leq r$ and $d<X$. 

\begin{theorem}
 Let $E$ be an elliptic curve over $\Q$ with good ordinary reduction at $2$. Assume that $N_E$ is squarefree and that $E(\Q)[2]=0$. Setting $\lambda:=\lambda_2(E/\Q)$, we find that
    \[n_{E, \leq \lambda}^{\op{alg}}(X)\gg \frac{X}{(\log X)^\delta}\] where $\delta:=\mathfrak{d}(\Omega)=\begin{cases}
        \frac{1}{3}\text{ if }\op{image}\bar{\rho}_{E, 2}\simeq G_3,\\
        \frac{2}{3}\text{ if }\op{image}\bar{\rho}_{E, 2}=\op{GL}_2(\F_2).\\
    \end{cases}$.
\end{theorem}

\begin{proof}
 According to Corollary \ref{cor 3.9}, if $d>0$ is a product of primes $\ell_i\in \Omega'$ then 
    \[\ralg(E^{(d)})\leq \op{corank}_{\Z_2} \op{Sel}_{2^\infty}(E^{(d)}/\Q)\leq \lambda_2(E/\Q).\]
    Let $\widetilde{\Omega}$ be the complement of $\Omega'$, i.e. it is the union of $\Omega$ and the primes $\ell|2 N_E$. Clearly, $\widetilde{\Omega}$ has the same density as $\Omega$. It then follows from Proposition \ref{density propn for n_Omega} that 
    \[n_{E, \leq \lambda}^{\op{alg}}(X)\gg n_{\widetilde{\Omega}}(X)\gg \frac{X}{(\log X)^\delta}.\]
    This completes the proof.
\end{proof}

\begin{proposition}\label{corank 1 lemma}
    Suppose that $E$ is an elliptic curve over $\Q$ and assume that the following assertions are satisfied:
  \begin{enumerate}
      \item $E$ has good ordinary reduction at $2$, 
      \item $E(\Q)[2]=0$,
      \item $\mu_2(E/\Q)=0$ and $\lambda_2(E/\Q)\leq 2$.
      \item There is a finite set of primes $\ell_1, \dots, \ell_k$ coprime to $2N_E$ such that 
      \begin{itemize}
          \item $\ell_1, \dots, \ell_k\notin \Omega$,
          \item $d=\ell_1\dots \ell_k\equiv 1\pmod{4}$.
          \item If $\omega(E)=+1$ (resp. $\omega(E)=-1$) then an odd (resp. even) number of the primes $\ell_i$ are inert in $\Q(\sqrt{-N_E})$ and the rest are split in $\Q(\sqrt{-N_E})$.
      \end{itemize} 
  \end{enumerate}
  Then we have that 
  \begin{enumerate}
      \item $\mu_2(E^{(d)}/\Q)=0$ and $\lambda_2(E^{(d)}/\Q)=1$, 
      \item $\omega(E^{(d)})=-1$, 
      \item $\op{corank}_{\Z_2}\left( \op{Sel}_{2^\infty}(E^{(d)}/\Q)\right)=1$.
  \end{enumerate}
\end{proposition}
\begin{proof}
The relationship between $\omega(E)$ and $\omega(E^{(d)})$ is given by 
\[\omega(E^{(d)})=\chi_d(-N_E) \omega(E),\] cf. \cite[section 4]{rubinsilverberg}. The character $\chi_d(-N_E)$ is given by the Kronecker symbol $\left(\frac{d}{-N_E}\right)$. Then we have that 
\[\left(\frac{\ell_1\dots\ell_k}{-N_E}\right)=(-1)^{(\frac{d-1}{2})(\frac{N_E-1}{2})}\left(\frac{-N_E}{\ell_1\dots\ell_k}\right)=\left(\frac{-N_E}{\ell_1\dots\ell_k}\right)=\left(\frac{-N_E}{\ell_1}\right)\cdots \left(\frac{-N_E}{\ell_k}\right).\]
Consequently, we have that $\omega(E^{(d)})=-1$.

\par We note that $E(\Q)[2]=0$ is a condition equivalent to requiring that $\rho_{E, 2}(\op{G}_{\Q})$ is either $G_3$ or all of $\op{GL}_2(\F_2)$. On the other hand, since $\Q(\sqrt{d})$ is disjoint from $\Q(E[2])$ since the primes $\ell_i$ are ramified in $\Q(\sqrt{d})$ and unramified in $\Q(E[2])$. It follows that $\rho_{E, 2}(\op{G}_{\Q(\sqrt{d})})=\rho_{E, 2}(\op{G}_{\Q})$. Thus, we find that $\rho_{E, 2}(\op{G}_{\Q(\sqrt{d})})$ is either $G_3$ or all of $\op{GL}_2(\F_2)$ and deduce that $E(\Q(\sqrt{d}))[2]=0$. In particular, it follows that $E^{(d)}(\Q)[2]=0$. Then Theorem \ref{dok-dok thm} of Dokchitser and Dokchitser applies to $E^{(d)}$ and implies that $\op{corank}_{\Z_2}\op{Sel}_{2^\infty}(E^{(d)}/\Q)$ is odd. Since $\ell_i\notin \Omega$ for $i\geq 2$, we find that 
\[\op{corank}_{\Z_2}\op{Sel}_{2^\infty}(E^{(d)}/\Q)\leq \lambda_2(E^{(d)}/\Q)=\lambda_2(E/\Q)\leq 2.\]
This in particular implies that $\op{corank}_{\Z_2}\op{Sel}_{2^\infty}(E^{(d)}/\Q)=1$.
\end{proof}

\begin{theorem}\label{n'E,1}
    Let $E$ be an elliptic curve over $\Q$ and assume that the following conditions are satisfied:
    \begin{enumerate}
        \item $E$ has good ordinary reduction at $2$ with squarefree conductor $N_E$, 
        \item $\omega(E)=-1$,
      \item $E(\Q)[2]=0$,
      \item $\mu_2(E/\Q)=0$ and $\lambda_2(E/\Q)\leq 2$.
    \end{enumerate}
    Then we find that 
    \[n_{E, 1}'(X)\gg \frac{X}{(\log X)^{\frac{11}{12}}}.\]
\end{theorem}
\begin{proof}
    Let $\mathcal{M}$ consist of primes $\ell$ such that $\ell\nmid 2 N_E$, $\ell\notin \Omega$, and $\ell$ splits completely in $F:=\Q(i, \sqrt{-N_E})$. Suppose that $d=\ell_1\dots \ell_s$ is a product of distinct primes $\ell_1, \dots, \ell_s\in \mathcal{M}$. Since each of the primes $\ell_i\equiv 1\pmod{4}$ it follows that $d\equiv 1\pmod{4}$. It then follows from Proposition \ref{corank 1 lemma} that $\op{corank}_{\Z_2}\left( \op{Sel}_{2^\infty}(E^{(d)}/\Q)\right)=1$. Let $\mathcal{M}^c$ be the complement of $\mathcal{M}$ and recall that $n_{\mathcal{M}^c}(X)$ is the number of squarefree $d<0$ that are products of primes $\ell\in \mathcal{M}$. We find that $n_{E, 1}'(X)\geq n_{\mathcal{M}^c}(X)$. On the other hand, Proposition \ref{corank 1 lemma} implies that 
    \[n_{\mathcal{M}^c}(X)\gg \frac{X}{(\log X)^{1-\mathfrak{d}(\mathcal{M})}}.\] 

    \par Thus what remains is to estimate $\mathfrak{d}(\mathcal{M})$. Let $k\subset \Q(E[2])$ be the field over which $\Q(E[2])/k$ is a cubic extension. Thus, if the residual image is conjugate to $G_3$, $k=\Q$. According to Remark \ref{trivial remark} a prime $\ell\nmid 2 N_E$ is not contained in $\Omega$ if it splits in $k$ and the primes above it are nonsplit in the extension $\Q(E[2])/k$. We show that this is independent of the condition that $\ell$ splits completely in the biquadratic extension $F$. We shall set $F(E[2])$ to denote the composite $F\cdot \Q(E[2])$. Since $[F\cdot k:k]$ is prime to $3$, we find that $\Q(E[2])\cap F\cdot k=k$. Consequently, we find that $\op{Gal}(F(E[2])/F\cdot k)\simeq \op{Gal}(\Q(E[2])/k)$. Let $\mathcal{S}$ denote the two nontrivial elements of $\op{Gal}(F(E[2])/F\cdot k)$ and view $\mathcal{S}$ as a subset of $\op{Gal}(F(E[2])/\Q)$. We find that $\ell\in \mathcal{M}$ if and only if $\op{Frob}_\ell\in \mathcal{S}$. Thus by the Chebotarev density theorem, 
    \[\mathfrak{d}(\mathcal{M})=\frac{\# \mathcal{S}}{[F(E[2]):\Q]}\geq \frac{2}{[\Q(E[2]):\Q][F:\Q]}\geq \frac{1}{12}.\]
    Thus we find that $1-\mathfrak{d}(\mathcal{M})\leq \frac{11}{12}$, which proves the result.
\end{proof}
The proof of the above result also shows that there is an explicit set of prime numbers $\mathcal{M}$ of density $\geq \frac{1}{12}$ such that for all $\ell\in \mathcal{M}$, $\op{corank}_{\Z_2}\left( \op{Sel}_{2^\infty}(E^{(\ell)}/\Q)\right)=1$.
\begin{theorem}\label{them prime twist 3.13}
    Let $E$ be an elliptic curve satisfying the conditions of Theorem \ref{n'E,1}. Let $\mathcal{M}$ be the set of prime numbers $\ell$ such that 
    \begin{itemize}
        \item $\ell\equiv 1\pmod{4}$, 
        \item $\widetilde{E}(\F_\ell)[2]=0$, 
        \item $\ell$ splits completely in $\Q(i, \sqrt{-N_E})$.
    \end{itemize}
    Then $\mathcal{M}$ has density $\geq \frac{1}{12}$ and $\op{corank}_{\Z_2}\left( \op{Sel}_{2^\infty}(E^{(\ell)}/\Q)\right)=1$.
\end{theorem}
\begin{proof}
    The result follows from the proof of Theorem \ref{n'E,1}.
\end{proof}
\par Before moving forward, let's illustrate Theorem \ref{n'E,1} through an explicit example.
\begin{example} Consider the elliptic curve with Cremona label \href{https://www.lmfdb.org/EllipticCurve/Q/53/a/1}{53a1} defined by 
\[E:y^2=x^3+405x+16038.\]
The data on LMFDB tells us that $E$ has good ordinary reduction at $2$ and $E(\Q)[2]=0$ with 
\[\mu_2(E/\Q)=0\text{ and }\lambda_2(E/\Q)=1.\]
This curve has Mordell--Weil rank $1$ and its root number is $-1$. Therefore Theorem \ref{n'E,1} applies to $E$ to give us that \[n_{E, 1}'(X)\gg \frac{X}{(\log X)^{\frac{11}{12}}}.\]
Proposition \ref{corank 1 lemma} asserts that if $\ell\equiv 1\mod{4}$ is a prime number which splits completely in $\Q(\sqrt{-53})$ and $2\nmid a_\ell(E)$, then $\op{corank}_{\Z_2}\left( \op{Sel}_{2^\infty}(E^{(d)}/\Q)\right)=1$. For instance, the primes $\ell=13, 17,$ and $29$ each satisfy these conditions. The same is true for the product $d:=13\times 17\times 29$.
\end{example}
We also prove a surprising result for twists by prime numbers for elliptic curves $E/\Q$ for which $E(\Q)[2]\neq 0$.

\begin{theorem}\label{pos density of primes}
    Let $E/\Q$ be an elliptic curve satisfying the following conditions:
    \begin{enumerate}
        \item $E$ has squarefree conductor $N_E$, 
        \item $E$ has good ordinary reduction at $2$ 
        \item $\omega(E)=+1$,
        \item $\mu_2(E/\Q)=0$ and $\lambda_2(E/\Q)=0$, 
        \item $E(\Q)[2]\neq 0$.
    \end{enumerate}
    Let $\ell$ be a prime satisfying the following conditions:
    \begin{itemize}
        \item $\ell\nmid 2 N_E$, 
        \item $\ell$ is inert in $K:=\Q(\sqrt{-N_E})$,
        \item $\ell\equiv 3, 5\pmod{8}$.
    \end{itemize}
Then we have that $\op{corank}_{\Z_2} \op{Sel}_{2^\infty}(E^{(\ell)}/\Q)=1$. The set of primes satisfying the above conditions have density $\frac{1}{4}$.
\end{theorem}
\begin{proof}
    Since $\ell\nmid 2 N_E$, we find that \[\omega(E^{(d)})=\chi_d(-N_E) \omega(E)=-\omega(E)=-1.\] Theorem \ref{dok-dok thm} of Dokchitser and Dokchitser asserts that $s_2(E)=\op{corank}_{\Z_2}\op{Sel}_{2^\infty}(E^{(\ell)}/\Q)$ is odd. Note that $E(\Q)[2]\neq 0$ implies that $\ell\in \Omega$, see Remark \ref{trivial remark}. Since $\ell\equiv 3, 5, 11, 13\pmod{16}$, we find that $n_\ell=\mathrm{ord}_2 \left( \frac{\ell^2 - 1}{8} \right)=0$ (in accordance with notation in the statement of Theorem \ref{thm:kida-matusno-2}). Thus it follows from Theorem \ref{thm:kida-matusno-2} that 
    \[\lambda_2(E^{(d)}/\Q) = \lambda_2(E/\Q) +  2^{n_\ell + 1}=2.\]
    It then follows from \eqref{effective bounds after Matsuno} that $\op{corank}_{\Z_2}\op{Sel}_{2^\infty}(E^{(\ell)}/\Q)\leq 2$. Thus, we have shown that \[\op{corank}_{\Z_2}\op{Sel}_{2^\infty}(E^{(\ell)}/\Q)=1.\]
    Note that there is no elliptic curve with conductor $2$. Since $N_E$ is squarefree, it must be therefore be divisible by some odd prime number. Thus, $K$ is linearly disjoint from $\Q(\sqrt{2})$. The condition that $\ell\equiv 3,5\pmod{8}$ is equivalent to $\ell$ being inert in $\Q(\sqrt{2})$. That the set of primes $\ell$ has density $\frac{1}{4}$ follows from the Chebotarev density theorem.
\end{proof}
\begin{example}
   Consider the elliptic curve with Cremona label \href{https://www.lmfdb.org/EllipticCurve/Q/15/a/4}{15a7} defined by $E:y^2=x^3-103707x+12854646$. The conditions of Theorem \ref{pos density of primes} are satisfied for $E$. The table below lists the ranks of the twists $E^{(\ell)}$ for the primes $7 \leq \ell \leq 47$, and the highlighted rows correspond to those $\ell$ which satisfy the bulleted conditions of Theorem~\ref{pos density of primes}.

\begin{table}[h!]
\begin{tabular}{|c|c|c|c|}
\hline
$\ell$                                    & $\ell$ mod $8$                           & decomposition in $K=\Q(\sqrt{-15})$        & $\ralg(E^{(\ell)})$                      \\ \hline \hline
$7$                                       & $7$                                      & inert                                      & $1$                                      \\ \hline
\cellcolor{green!25}$11$ & \cellcolor{green!25}$3$ & \cellcolor{green!25}inert & \cellcolor{green!25}$1$ \\ \hline
\cellcolor{green!25}$13$ & \cellcolor{green!25}$5$ & \cellcolor{green!25}inert & \cellcolor{green!25}$1$ \\ \hline
$17$                                      & $1$                                      & split                                      & $0$                                      \\ \hline
$19$                                      & $3$                                      & split                                      & $0$                                      \\ \hline
$23$                                      & $7$                                      & split                                      & $0$                                      \\ \hline
\cellcolor{green!25}$29$ & \cellcolor{green!25}$5$ & \cellcolor{green!25}inert & \cellcolor{green!25}$1$ \\ \hline
$31$                                      & $7$                                      & split                                      & $0$                                      \\ \hline
\cellcolor{green!25}$37$ & \cellcolor{green!25}$5$ & \cellcolor{green!25}inert & \cellcolor{green!25}$1$ \\ \hline
$41$                                      & $1$                                      & inert                                      & $1$                                      \\ \hline
\cellcolor{green!25}$43$ & \cellcolor{green!25}$3$ & \cellcolor{green!25}inert & \cellcolor{green!25}$1$ \\ \hline
$47$                                      & $7$                                      & split                                      & $0$                                      \\ \hline
\end{tabular}
\label{table: rank 1 examples}
%\caption{An illustration of Theorem~\ref{pos density of primes} for the elliptic curve $15a7$ and the primes below $50$.}
\end{table}
\end{example}

\section{Prescribed $\lambda$-invariants}\label{s 4}
\par In this short section, we prove a result about the distribution of $\lambda$-invariants in quadratic twist families. We fix an elliptic curve $E_{/\Q}$ with good ordinary reduction at $2$ with squarefree conductor $N_E$. Moreover we shall assume that $E(\Q)[2]=0$. Consider the family of quadratic twists $E^{(d)}$ where $d>0$ is a squarefree integer and let $m_{E, N}(X)$ count the number of positive squarefree $d\leq X$ such that $\lambda(E^{(d)}/\Q)=N$. 

\begin{theorem}\label{last thm}
    Assume that $\mu_2(E/\Q)=0$. If $N$ is any integer such that $N\geq \lambda_2(E/\Q)$ and $N\equiv \lambda_2(E/\Q)\pmod{2}$, then we have that 
    \[m_{E, N}(X)\gg X/(\log X)^{\alpha},\]
    where $\alpha=\begin{cases}
        \frac{1}{3}
        & \text{ if }\op{Gal}(\Q(E[2])/\Q)\simeq \Z/3\Z;\\
        \frac{2}{3}& \text{ if }\op{Gal}(\Q(E[2])/\Q)\simeq S_3.
    \end{cases}$
\end{theorem}
\begin{proof}
    Let $\mathcal{Q}$ consist of primes $\ell \nmid 2 N_E$ such that:
    \begin{itemize}
        \item $\ell\equiv 3, 5\pmod{8}$, 
        \item $\ell\in \Omega$, i.e., $\widetilde{E}(\F_\ell)[2]\neq 0$. 
    \end{itemize}
    Let $k$ be the subfield of $\Q(E[2])$ such that $\Q(E[2])/k$ is a cubic extension. Suppose that $\ell\nmid 2 N_E$ is a prime which is inert in $\Q(\sqrt{2})$ and any prime $v$ of $k$ that lies above $\ell$ is completely split in $\Q(E[2])$. Then, $\ell\in \mathcal{Q}$ and thus $\mathcal{Q}$ has positive density. Write $N=\lambda_2(E/\Q)+2k$ and pick $k$ primes $q_1, \dots, q_k\in \mathcal{Q}$. Since $q_i\equiv 3,5 \pmod{8}$ we find that $n_{q_i}=0$. Let $d':=q_1\dots q_k$ and let $\ell_1,\dots ,\ell_s$ be primes such that $\ell_i\nmid 2N_E$ and $\ell_i\notin \Omega$. We then set $d:=d' d''$, where $d'':=\ell_1\dots\ell_s$. From \eqref{thm:kida-matusno-2} we deduce that 
    \[\lambda_2(E^{(d)}/\Q) = \lambda_2(E/\Q) + \sum_{\substack{\ell \mid d \\ 2 \mid \# \tilde{E}(\F_2)}} 2^{n_\ell + 1}=\lambda_2(E/\Q) +2k=N.\]
    Keeping $d'$ fixed, the number of choices for $d''$ such that $d''\leq \frac{X}{d'}$ is $n_{\Omega}(X/d')$. By Proposition \ref{density propn for n_Omega}, there is a constant $c>0$ such that
    \[n_{\Omega}(X/d')\sim  \frac{c(X/d')}{\left(\log (X/d')\right)^{\mathfrak{d}(\Omega)}}.\]
    Relabeling $c$, we find that 
    \[n_{\Omega}(X/d')\sim c X/(\log X)^{\mathfrak{d}(\Omega)}.\]
    Thus \[m_{E,N}(X)\geq n_{\Omega}(X/d')\gg  X/(\log X)^{\delta},\]
    where by Corollary \ref{cor:omega-density-p-2}, \[\alpha:=\mathfrak{d}(\Omega)=\begin{cases}
        \frac{1}{3}
        & \text{ if }\op{Gal}(\Q(E[2])/\Q)\simeq \Z/3\Z;\\
        \frac{2}{3}& \text{ if }\op{Gal}(\Q(E[2])/\Q)\simeq S_3.
    \end{cases}\]
\end{proof}
\begin{example}
    Consider the elliptic curve with Cremona label \href{https://www.lmfdb.org/EllipticCurve/Q/53/a/1}{53a1} defined by 
\[E:y^2=x^3+405x+16038.\]
From the data in LMFDB, we find that $\mu_2(E/\Q)=0$ and $\lambda_2(E/\Q)=1$. Moreover, the representation $\bar{\rho}_{E,2}$ is surjective. Theorem \ref{last thm} asserts that for any odd integer $N\geq 1$, 
\[m_{E, N}(X)\gg X/(\log X)^{2/3}.\]
\end{example}
\begin{example}
    Consider the elliptic curve \href{https://www.lmfdb.org/EllipticCurve/Q/17/a/4}{17a4} defined by 
    \[E: y^2=x^3-11x+6.\]
    According to LMFDB, $\mu_2(E/\Q)=0$ and $\lambda_2(E/\Q)=0$. The representation $\bar{\rho}_{E,2}$ is surjective. Theorem \ref{last thm} asserts that 
\[m_{E, N}(X)\gg X/(\log X)^{2/3}\]
for any even integer $N\geq 0$.
\end{example}

\bibliographystyle{abbrv}
\bibliography{references}
\end{document}